\documentclass[letterpaper, 10pt]{amsart}

\usepackage{mathptmx}
\usepackage[latin1]{inputenc}

\usepackage[alphabetic]{amsrefs}
\usepackage{setspace}
\usepackage{graphicx}
\usepackage[all]{xy}
\usepackage{wasysym}
\usepackage{multicol}

\usepackage{acronym}

\makeatletter
\@namedef{subjclassname@2010}{%
  \textup{2010} Mathematics Subject Classification}
\makeatother

\usepackage{caption}
\usepackage{subcaption}

\newcommand {\M}{\mathcal{C}}

\newcommand {\R}{\mathbb{R}}

\theoremstyle{plain}
\newtheorem{thm}{Theorem}[section]
\newtheorem{lem}[thm]{Lemma}
\newtheorem{prop}[thm]{Proposition}

\theoremstyle{definition}
\newtheorem{defn}[thm]{Definition}

\theoremstyle{remark}
\newtheorem{rem}[thm]{Remark}

\begin{document}

\baselineskip=17pt

\title[Plane Self-Similar Solutions] 
{A Note on Self-Similar Solutions of the Curve Shortening Flow}

%
\author{M\'arcio Rostirolla Adames}
\address[M\'arcio Rostirolla Adames]{Departamento Acad\^emico de Matem\'atica\\
        UTFPR - Campus Curitiba\\
        Curitiba - PR, Brasil 80230-901}
\email[M.~Adames]{marcioadames@utfpr.edu.br}

\date{}


\begin{abstract}
This article gives an alternative approach to the self-shrinking and self-expanding solutions of the curve shortening 
flow, which are related to singularity formation of the mean curvature flow. The motivation for the self-similar 
solutions arises from natural area preserving rescaling. Further we describe the self-similar solutions in terms of 
a simple ODE and give an alternative proof that they lie in planes. 
\end{abstract}



\maketitle

\begin{section}{Introduction}
To deform a curve (usually smooth) by the curve shortening flow (CFS) is to let it evolve in the direction of its 
curvature vector, thus generating a family of curves. The problem of understanding the behavior of such family was 
first addressed by Mullins \cite{MR0078836} in 1956 to study ideal grain boundary motion in two dimensions. This 
problem was further studied by Brakke \cite{MR485012} in 1978 (and 1975) in a more general setting (surfaces moving by 
their mean curvature). Renewed interest in the topic came with the works of Gage and Hamilton (e. g. \cite{MR840401}, 
where they show that convex plane curves shrink to a point, becoming more circular as time advances) and Grayson (e. g. 
\cite{MR979601}). Since then the problem has been studied by many, and of particular significance has been the study of 
singularity formation.

The present work does not purport to contain a comprehensive introduction to the CFS because of the great number 
of contributions to the subject (e. g. ``The curve shortening flow'' of Chou and Zhu \cite{MR1888641} contains 113 
items in it's bibliography). As an important result we cite the complete classification of closed plane curves which 
shrink under the CFS by Abresch and Langer \cite{MR845704}, which is related to the singularity formation of the flow 
(\cite{MR1030675} from Huisken for the Mean Curvature Flow, which generalizes the CSF for higher dimension).

This work was somewhat inspired by the recent works of Halldorsson \cite{MR2931330}, which classifies self-similar (in a 
broader context) plane curves of the CSF, and Altschuler, Altschuler, Angenent and Wu \cite{MR3043378}, which 
provides a classification of self-similar solutions (or solitons) of the CSF in $\R^n$. Both works are based 
in ODE techniques. The last of the works mentions the well known fact (see \cite{MR3043378}) that dilating solitons are 
planar and a we prove this fact sections 5 and 6.

Section 2 includes motivation for the CSF and for self-similar solutions, and a description of the 
problem as a PDE. In section 3 polar coordinates and a moving frame based in these coordinates are employed to 
characterize the plane self-shrinking solutions (known as Abresch \& Langer curves) in terms of solutions to a simple 
nonlinear ODE in Theorem \ref{secondary}. In section 4 spherical coordinates and a moving frame based on them are used 
to plot spatial self-shrinkers, and these plots indicate that the self-shrinkers lie in planes; result which is shown 
for curves in $\mathbb{R}^n$ at section 5. Section 6 brings the analogs of sections 4 and 5 for self-expanding curves.
\end{section}

\begin{section}{Basics of the flow}

Given a continuous family of smooth regular curves $\gamma_t:I\to \R^n$, defined on some interval $I$, we can see it 
as a two variable function $\gamma:(a, b) \times I \to \R^n$ defined by $\gamma(t,x)=\gamma_t(x)$.
\begin{defn} A family $\gamma:(a, b) \times I \to \R^n$ of smooth immersions $\gamma_t:I\to \R^n$, evolves by the 
\textit{curve shortening flow} (CSF) if it satisfies
\begin{equation}\label{key}
\left(\frac{\partial \gamma}{\partial t}\right)^\perp = \frac{\partial^2 \gamma}{\partial s^2},
\end{equation}
where $s$ is the arc length parameter (not necessarily the parameter of $I$) of the curve $\gamma_t$.
\end{defn}

This flow has been studied by many and we deduce some fundamental results in this section.

In order to see a motivation for this flow, let us consider some arbitrary family of immersions: Let $\gamma_0: I 
\to \R^n$ be an immersion and $\gamma: (-\epsilon, \epsilon) \times  I \to \R^n$ be a smooth family of immersions 
$\gamma_t:I \to \R^n$ with $\gamma(0,u) = \gamma_0(u)$, for all $u \in I$. Further let us denote $V(u)$ for the 
\textit{infinitesimal generator} of the variation $\gamma_t$
$$
V(u) = \left.\frac{\partial \gamma}{\partial t}(t,u) \right|_{t=0}.
$$

Let us denote $\M$ for the smooth curve $\gamma_0(I)$ and assume, w.l.o.g., that $\gamma_0$ is a 
parametrization of $\M$ by 
arc-length. If $\M$ is closed then the length $\mathcal{L}_t$ of $\gamma_t(I)$ induced by $\gamma_t$ is finite for 
any $t \in (-\epsilon, \epsilon)$ and we can calculate:
\begin{prop}
The first variation of the arc length with respect to $\gamma$ is given by:
$$
\left.\frac{d \mathcal{L}_t}{d t}\right|_{t=0} =-\int_0^L \left<V(s), \frac{\partial^2 
\gamma_0}{\partial s^2}(s) \right> ds.
$$
\end{prop}\label{LengthDecrease}
\begin{proof} 
It holds that $\gamma_0(0)=\gamma_0(L)$ and
\begin{align*}
\left.\frac{d \mathcal{L}_t}{d t}\right|_{t=0} =& \int_0^L\left.\frac{d}{dt}|\gamma_t(s)|\right|_{t=0}ds = 
\int_0^L\frac{1}{|\gamma_0(s)|} \left<\left.\frac{\partial^2}{\partial t \partial s}\gamma_t(s)\right|_{t=0}, 
\frac{\partial \gamma_0}{\partial s}(s) \right>ds\\
=& -\int_0^L \left<V(s), \frac{\partial^2 \gamma_0}{\partial s^2}(s) \right> ds,
\end{align*}
where $s$ is only meant to be the arc-length parameter at $t=0$.
\end{proof}

But $\frac{\partial^2 \gamma_0}{\partial s^2}(s) \perp \M$ at $\gamma_0 (s)$. So that only the normal component of 
$V(s)$ changes the length of $\M$. Furthermore the length is not changed if $V(s)$ is tangent to the surface for all 
$s \in I$, which is obvious, since $V(s)$ tangent would mean that the length's variation at $t=0$ is the same 
as the one made by the parameter change
$$
\gamma(t,s) = \gamma_0\left(t\left<V(s), \frac{\partial \gamma_0}{\partial s}(s)\right>+s\right).
$$

The vector field $-\frac{\partial^2 \gamma_0}{\partial s^2}(s)$ can be seen as the \textit{gradient} of the 
length functional. This means that deforming a curve by the curve shortening flow is to decrease the length of the curve 
in the fastest possible direction among all infinitesimal variations with $L^2$-norm equal to 
$$
C:=\left(\int_0^l\left<\frac{\partial^2 \gamma_t}{\partial s^2},\frac{\partial^2 \gamma_t}{\partial s^2}\right> 
d\mu\right)^{1/2}.
$$

To justify the statement about the fastest possible deformation direction, notice that Prop. \ref{LengthDecrease} 
implies that the infinitesimal generators causing the greatest, in absolute value, variation in the 
length are the ones in the direction of $\frac{\partial^2 \gamma_0}{\partial s^2}(s)$, i. e.
$$
V(\gamma_0(s)) = a(s) \frac{\partial^2 \gamma_0}{\partial s^2}(s).
$$
Thus the maxima and minima are critical points of the functional
$$
I(a):=-\int_0^L a(s)\left<\frac{\partial^2 \gamma_0}{\partial s^2}(s), \frac{\partial^2 \gamma_0}{\partial 
s^2}(s) \right> ds
$$
subject to the restriction
$$
H(a) =\int_0^L a^2(s)\left<\frac{\partial^2 \gamma_0}{\partial s^2}(s), \frac{\partial^2 \gamma_0}{\partial s^2}(s) 
\right> ds=C^2.
$$
Therefore, from the Euler-Lagrange equations for this problem (for example Courant and John
\cite{MR1016380}), one gets
$$
-\|\gamma''_0(s)\|^2+2\lambda a(s) \|\gamma''_0(s)\|^2=0,
$$
so that $a(s)$ is constant and then $a(s)=1$ (because of the restriction).

Let us follow considering the variation of the area limited by the curve (for closed plane curves) with respect to an 
arbitrary family of immersions.

In this way the target space is $\R^2$ and one can write
$$
\gamma_t(s)= (\gamma_t^1(s), \gamma_t^2(s)),
$$
so that the tangent vector and a normal versor, respectively, are
$$
\frac{\partial\gamma_t(s)}{\partial s}= \left(\frac{\partial\gamma_t^1}{\partial s}(s), 
\frac{\partial\gamma_t^2}{\partial s}(s)\right), \hspace{8mm} N_t(s)= 
\left(-\frac{\partial\gamma_t^2}{\partial s}(s), \frac{\partial\gamma_t^1}{\partial s}(s)\right),
$$
and the infinitesimal generator
$$
 V_t(s)= \frac{\partial\gamma_t}{\partial t}(s) = (V^1(s), V^2(s)).
$$

Further the area limited by $\gamma_t(s)$ is given, from Green's Theorem, by
$$
\mathcal{A}(\gamma_t(s)) = \frac{1}{2}\int_0^L (-\gamma_t^2(s), \gamma_t^1(s))\cdot \frac{\partial\gamma_t}{\partial 
s}(s) ds
$$

Therefore the first variation of the area can be calculated to be
\begin{align*}
 \frac{d}{dt}\mathcal{A}(\gamma_t) =& \frac{1}{2}\int_0^L \frac{\partial}{\partial t}(-\gamma_t^2(s), 
\gamma_t^1(s))\cdot \frac{\partial\gamma_t}{\partial s}(s) ds + \frac{1}{2}\int_0^L (-\gamma_t^2(s), 
\gamma_t^1(s)\cdot\frac{\partial^2\gamma_t}{\partial t \partial s}(s) ds\\
=& \frac{1}{2}\int_0^L (-V^2(s), V^1(s))\cdot \frac{\partial\gamma_t}{\partial s}(s) ds - \frac{1}{2}\int_0^L 
\frac{\partial}{\partial s}(-\gamma_t^2(s),\gamma_t^1(s))\cdot\frac{\partial\gamma_t}{\partial t}(s) ds\\
=& \int_0^L \left<V_t(s), N_t(s)\right> ds,
\end{align*}
which, of course, is zero if $V_t(s)$ is parallel to the tangent vector of $\gamma_t(s)$ for all $s \in I$. 
Furthermore, among all 
vector fields $V_t$ with given $L^2$ norm, the ones causing greater change in the area are pointing everywhere in the 
normal direction.

With the purpose of deforming the initial curve \textit{and} still preserve area, one could rescale the immersions 
taking another family of immersions:
$$
 \tau_t(s):=c(t)\gamma_t(s).
$$
Then the rate of change of the area enclosed by $\tau_t$ is
$$
 \frac{d}{dt}\mathcal{A}(\tau_t) =\int_0^L c(t)\frac{dc}{dt}(t)\left<\gamma_t(s), N_t(s)\right> + 
c^2(t)\left<V_t(s), N_t(s)\right> ds
$$

Searching for constant enclosed area, it is then natural to look for families of immersions satisfying 
\begin{equation}
\label{varia} V_t(s) = -\frac{c'(t)}{c(t)}\gamma_t^\perp(s).
\end{equation}

Only some very particular variations can have infinitesimal generators satisfying eq. \eqref{varia}. This 
happens because $V_t(s)$ has to point in the direction of $\gamma_t^\perp(s)$ and it is necessary that, fixing some 
$t$, there is a constant $\frac{c'(t)}{c(t)}$ scaling $V_t(s)$ to $-\gamma_t^\perp(s)$ for all $s \in I$.

It is the case that some families of immersions evolving by the curve shortening flow naturally have this 
property. Further such a family of immersions is just a homothety (see below). This means that the shape of the curve 
(which defines $\frac{\partial^2 \gamma}{\partial s^2}$ and $-\gamma^\perp(s))$) is of greater importance than the 
function $\frac{c'(t)}{c(t)}$, which only changes the size (and/or reflects) the curve. As we are interested in the 
shape of the curve, we fix $t \in [0,T)$ and rescale $\gamma_t(s)$ in order to cancel out  the constant
$\frac{c'(t)}{c(t)}$. In the curve shortening flow, $V_t = \frac{\partial^2 \gamma_t}{\partial s^2}$ and eq. 
\eqref{varia} becomes
\begin{equation}
\label{shrinker1} \frac{\partial^2 \gamma_t}{\partial s^2} = -\gamma_t^\perp(s),
\end{equation}
Such immersions are called \textit{shrinking self-similar solutions} (or \textit{self-shrinkers}) of the curve 
shortening flow. More important than this motivation, is the fact that self-shrinkers are related to the singularity 
formation of the flow, which is not in the present work. The interested reader could consult \cite{MR1030675}.

\begin{thm}\label{hint}
If $\M$ is a closed smooth curve in $\R^n$ and $\gamma: I \to \M$ is a parametrization of $\M$ satisfying 
$$
\frac{\partial^2 \gamma}{\partial s^2} = -\gamma^\perp(s),
$$
then the homothety given by the family of immersions
$$
\varphi(t,s):= \sqrt{1-2t} \gamma(s),
$$
with $s$ being the arc-length parameter only at $t=0$, evolves by the curve shortening flow and satisfies 
$\varphi(0,s):= \gamma(s)$.
\end{thm}
\begin{proof}
 In fact it is a solution to \eqref{key} because rescaling $\gamma(s)$ has the effect of dividing $\frac{\partial^2 
\gamma}{\partial s^2}$ by the same factor and
 $$
\left(\frac{\partial \varphi(t,s)}{\partial t} \right)^{\perp}= \left(\frac{-1}{\sqrt{1-2t}}\gamma(s) 
\right)^{\perp}=\frac{\partial^2 \gamma}{\partial  s^2}
 $$
\end{proof}

There has been great interest in studying these self-shrinkers (also in other contexts), for they are related to the 
singularities of the curve shortening flow. A classification is due to Abresch \& Langer \cite{MR845704}.

\begin{subsection}{PDE point of view}
 
With equation \eqref{key} we defined a family of immersions evolving by the curve shortening flow, however it is more 
natural to have an initial smooth immersion $\gamma_0: I \to \R^n$ and try to deform it by the 
CSF, i. e. to find a family of immersions $\gamma:[0, T) \times I \to \R^n$ evolving by the CSF with 
$\gamma(0,\cdot)=\gamma_0(\cdot)$. Note that 
finding this family is \textit{not}, in the form described, to solve the initial value problem:
\begin{align*}
 \left(\frac{\partial \gamma}{\partial t}\right)^\perp =& \frac{\partial^2 \gamma}{\partial s^2},\\
 \gamma(0,\cdot) =& \gamma_0(\cdot),
\end{align*}
then there is no guarantee that all the immersions of the family are parametrizations by arc-length (even if the 
initial one is so parametrized) of the curves defined by them. But the acceleration vector above can be written, in an 
arbitrary parametrization $\gamma_t(u)$ as
$$
\frac{\partial^2 \gamma}{\partial s^2} = \frac{\gamma''_t(u)}{\left<\gamma'_t(u), \gamma'_t(u)\right>} - 
\frac{\left<\gamma''_t(u), \gamma'_t(u) \right>}{\left(\left<\gamma'_t(u), \gamma'_t(u)\right>\right)^2}\gamma'_t(u).
$$

This would then lead us to the initial value problem:
\begin{equation}
\label{original}\left\{
\begin{array}{l}
 \left(\frac{\partial \gamma}{\partial t}\right)^\perp = \frac{\gamma''_t(u)}{\left<\gamma'_t(u), 
\gamma'_t(u)\right>} - \frac{\left<\gamma''_t(u), \gamma'_t(u) \right>}{\left(\left<\gamma'_t(u), 
\gamma'_t(u)\right>\right)^2}\gamma'_t(u),\\
\gamma(0,\cdot) = \gamma_0(\cdot).
\end{array}
\right.
\end{equation}
 
In order to better understand this formulation of the problem, let us consider the related problem:
\begin{equation}\label{var1}
\left\{
\begin{array}{l}
 \frac{\partial \gamma}{\partial t} = \frac{\gamma''_t(u)}{\left<\gamma'_t(u), 
\gamma'_t(u)\right>} - \frac{\left<\gamma''_t(u), \gamma'_t(u) \right>}{\left(\left<\gamma'_t(u), 
\gamma'_t(u)\right>\right)^2}\gamma'_t(u),\\
\gamma(0,\cdot) = \gamma_0(\cdot).
\end{array}
\right.
\end{equation}

Note that any solution of \eqref{var1} is also a solution to \eqref{original} because the vector on the right side is 
already perpendicular to the curve defined by $\gamma_t$. On the other hand a solution $\gamma(t,u)$ to 
\eqref{original} 
induces a solution $\beta(t,u)$ of \eqref{var1} in the following way: Let
$$
\psi(t,u):=u-\int_0^t\frac{\left<\frac{\partial \gamma}{\partial t}, 
\frac{\partial \gamma}{\partial u}\right>}{\left< \frac{\partial \gamma}{\partial u},\frac{\partial \gamma}{\partial 
u}\right>}dt.
$$
Thus $\psi(0,u)=u$ and $\beta(t,u):=\gamma (t,\psi (t,u))$ is a solution to \eqref{var1} because of the coordinate 
independence of $\frac{\partial^2 \gamma}{\partial s^2}$.

This means that solutions to eq. \eqref{var1} are in fact solutions to eq. \eqref{original}  parametrized in a 
particular fashion. As we are interested into the geometric properties, the parametrization of the curve must play no 
role, thus the curve shortening flow is stated as in eq. \eqref{original}. Nevertheless the problem stated in eq. 
\eqref{var1} is an initial value problem more suitable to be studied with help of PDE theory.

\begin{rem}
One might be missing further boundary conditions in eqs. \eqref{var1} or \eqref{original}, specifically conditions at 
the boundary of $I$. As we are interested in the curve as geometric object, it is meant to be extended as long as 
possible so that often $I = \R$. This has the side effect of entire sections of $\M_t:=\gamma_t(I)$ being run over and 
over by $\gamma_t$ if the curve is closed. This could be mended by taking $I=[a,b]$ and further requiring
$$
\gamma(a)=\gamma(b), \hspace{3mm} \gamma'(a)=\gamma'(b), \hspace{3mm} \gamma''(a)=\gamma''(b), \,\,\, \ldots
$$
\end{rem}

In the literature it is often considered a fixed curve $\M$ and family of immersions $\gamma: \M \to \R^n$, so that eq. 
\eqref{var1} is only a local form for the initial value problem. For further reading about existence an uniqueness of 
solutions to the curve shortening flow we indicate Gage and Hamilton's paper \cite{MR840401}, Grayson's article 
\cite{MR979601} or Chou and Zhu's book \cite{MR1888641}. For solutions of such systems in a more general setting 
could be consulted \cite{SMK11}, \cite{MR2815949} or \cite{MR2274812}, among others.

\end{subsection}

\end{section}

\begin{section}{Plane self-shrinkers.}

In Theorem \ref{hint} it is described how the curve shortening flow acts on a self-shrinkers, but not how to find such 
special solutions nor the properties of them. In this section we forget the flow and search for static solutions, so 
that we do not have families but a single immersion satisfying eq. \eqref{shrinker1}.

Let $\gamma: I \to \R^2$ be a self-similar shrinking solution of the curve shortening flow that is 
parametrized by arc-length. Thus
\begin{equation}
\label{selfshrinker} \gamma'' = - \gamma^\perp = \langle \gamma, \gamma' \rangle \gamma' - \gamma.
\end{equation}

\begin{lem}
The only self-shrinkers (solution of eq. \eqref{selfshrinker}) that pass through the origin are the straight lines.
\end{lem}
\begin{proof}
If $\gamma(t_0) = 0$ and $\gamma'(t_0)=\overrightarrow{v}$, then $\|\overrightarrow{v}\|=1$, for the curve is 
parametrized by 
arc-length. It follows that $\beta(t) = (t-t_0)\overrightarrow{v}$ satisfies
$$
\beta''(t)=0,
$$
and
$$
 \langle \beta, \beta' \rangle \beta' - \beta = (t-t_0)\overrightarrow{v} - (t-t_0)\overrightarrow{v}=0.
$$
Therefore $\beta(t)$, $t\in \R$, is a solution to \eqref{selfshrinker} with $\beta(t_0)=\gamma(t_0)=0$ and 
$\beta'(t_0)=\gamma'(t_0)=\overrightarrow{v}$. From the uniqueness of the solutions to the associated (with eq.
\eqref{selfshrinker}) initial value problem it follows that $\gamma(t)=\beta(t)$.
\end{proof}

The straight lines are static under the curve shortening flow. As the other solutions do not cross the origin, 
we can write them in polar coordinates. We follow calculating 
$$
\langle \gamma, \gamma\rangle'' = 2 \langle \gamma'',\gamma \rangle + 2 \langle \gamma', \gamma' \rangle,
$$
and, writing $\alpha = \langle \gamma, \gamma\rangle$, we get in view of eq. \eqref{selfshrinker}
\begin{equation}
\label{alpha}\alpha'' - \frac{(\alpha')^2}{2} +2 \alpha = 2.
\end{equation}

The associated initial value problem admits an unique solution. Further there are solutions of eq. \eqref{alpha} that 
are always positive:

\begin{lem}
A solution of eq. \eqref{alpha} with $0 < \alpha(0) < 1$ and $\alpha'(0)=0$ is strictly positive.
\end{lem}
\begin{proof}
First note that $\alpha(t)$ has a local minimum at $t=0$, so that if there is $t_1 \in D(\alpha)$ 
such that  
$\alpha(t_1)\leq \alpha (0)$, then there would be a local maximum at some $t_0 \in (0, t_1)$. But
$$
\beta(t) := \alpha(t_0 - t),
$$
also satisfies eq. \eqref{alpha} and $\beta(0) = \alpha(t_0)$, $\beta'(0) = \alpha'(t_0)$ and $\beta''(0) = 
\alpha''(t_0)$. Thus a solution of eq. \eqref{alpha} would exist for all $t \in \R$ and be given by
$$
\alpha^*(t) = \left\{\begin{array}{ll}
     \alpha(t- 2nt_0), &t\in [2nt_0, (2n+1)t_0)\\
     \beta(t- 2nt_0) = \alpha((2n+1)t_0 - t), &t\in [(2n+1)t_0, (2n+2)t_0)
                   \end{array}
 \right.
$$
so that min $\alpha(t) = \alpha(0)$.
\end{proof}

The figure below illustrates the construction a the solution of eq. \eqref{alpha}. Further the periodicity of the 
solution is expected from the actual form of the Abresch \& Langer curves.
\vspace{-9mm}
\begin{figure}[h]
 \label{solalpha}\includegraphics[height=50mm]{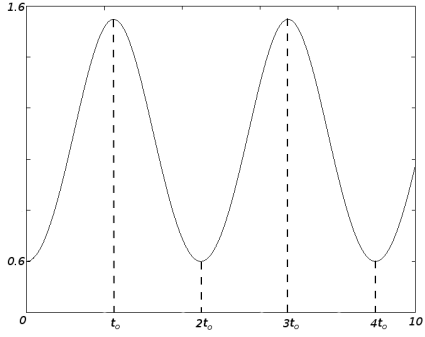}
  \caption{A solution $\alpha(t)$, with $\alpha(0) = 0.6$ and $\alpha'(0) = 0$}
\end{figure}

For every solution of eq. \eqref{alpha} that is positive, it is possible to define a function $u=\sqrt{\alpha}$ and 
write the self-shrinker in polar coordinates:
\begin{equation}
\label{polarplane}\gamma (t)= u(t)(\mbox{cos}(\theta(t)), \mbox{sin}(\theta(t))).
\end{equation}

Beyond this $u=\sqrt{\alpha}$ implies
$$
\alpha' = 2uu' \hspace{13mm} \mbox{and} \hspace{13mm} \alpha'' = 2u''u+2(u')^2
$$
so that equation \eqref{alpha} turns into
\begin{equation}
\label{uuu} u''u + (u')^2 - [u']^2u^2 + u^2 = 1.
\end{equation}

Further it holds
\begin{align}
\label{1}\gamma' =& u'(\mbox{cos}\theta, \mbox{sin}\theta) + u\theta'(-\mbox{sin}\theta, \mbox{cos}\theta)\\
\label{2}\gamma'' =& [u''-u[\theta']^2](\mbox{cos}\theta, \mbox{sin}\theta) + [2u'\theta'+u\theta''](-\mbox{sin}\theta, \mbox{cos}\theta)\\
\label{3}-\gamma^\perp =& [u[u']^2-u](\mbox{cos}\theta, \mbox{sin}\theta) + [u^2u'\theta'](-\mbox{sin}\theta, \mbox{cos}\theta),
\end{align}
Therefore equation \eqref{selfshrinker} holds if, and only if, both equations hold:
\begin{align}
\label{A} u''-u[\theta']^2 = u[u']^2-u,\\
\label{B} 2u'\theta'+ u\theta''=u^2u'\theta'
\end{align}
where $u$ is a known function and, recalling that $\|\gamma'\|=1$,
\begin{equation}
 \label{C} [\theta']^2 = \frac{1-[u']^2}{u^2}
\end{equation}
and
\begin{equation}
\label{thetasolution}\theta= \int \frac{4\alpha(t)-(\alpha'(t))^2}{4\alpha^2(t)}dt.
\end{equation}

In figure \ref{Abresch Langer} there are plots of self-shrinkers  constructed from numerical solutions of eqs. 
\eqref{alpha} and \eqref{thetasolution}. 
It is not clear which initial conditions generate closed curves.
\begin{figure}[h]
\includegraphics[width=\linewidth]{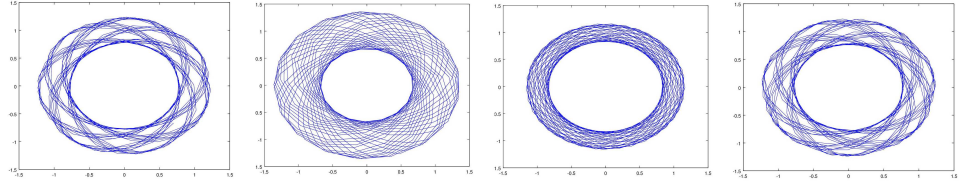}
\caption{Noncompact Abresch \& Langer Curves.}
 \label{Abresch Langer}
\end{figure}

It is not hard to see that any solutions $u$ and $\theta$ of equations \eqref{uuu} and \eqref{C} also satisfy eq. 
\eqref{A} and \eqref{B} and thus generate self-shrinkers of the curve shortening flow through equation 
\eqref{polarplane}. Thus we have proved:

\begin{thm}\label{secondary}
A curve $\M$ parametrized by $\gamma: I \to \R^2$, $\gamma(t)=\sqrt{\alpha(t)}(\mbox{cos}(\theta(t)), 
\mbox{sin}(\theta(t)))$ is a self-shrinker of the curve shortening flow if, and only if,
\begin{enumerate}
 \item it is a straight line or
 \item $\alpha(t) > 0$ for all $t \in I$ and 
\begin{align*}
&\alpha'' - \frac{(\alpha')^2}{2} +2 \alpha = 2,\\
&\theta= \int \frac{4\alpha(t)-(\alpha'(t))^2}{4\alpha^2(t)}dt.
\end{align*}
\end{enumerate}
\end{thm}
\end{section}

\begin{section}{Plane self-shrinkers.}
The advantage of the method in the previous section is that it can be generalized to spatial self-shrinking 
curves. Consider now a self-similar solution of the curve shortening flow $\gamma: I \to \R^3$ that is parametrized by 
arc-length, then $\alpha = \langle \gamma, \gamma\rangle$ also satisfies eq. \eqref{alpha}. Denoting $u=\sqrt{\alpha}$ 
and taking a positive solution O.D.E. \eqref{alpha} one can write the self-shrinker in spherical coordinates:
$$
\gamma (t) = u(\cos \theta(t) \sin \varphi(t), \sin \theta(t) \sin \varphi(t), \cos\varphi(t)).
$$
 
 We use the following moving frame to calculate $\gamma''$ and $\gamma^\perp$:
 $$
 X=\left(\begin{array}{c}
          \cos \theta \sin\varphi\\
          \sin \theta \sin\varphi\\
          \cos\varphi
         \end{array}
\right),
 \hspace{7mm}
 \frac{\partial X}{\partial \theta}=\left(\begin{array}{c}
          -\sin \theta \sin\varphi\\
          \cos \theta \sin\varphi\\
          0
         \end{array}
\right),
 \hspace{7mm}
\frac{\partial X}{\partial \varphi}=\left(\begin{array}{c}
          -\cos \theta \cos\varphi\\
          -\sin \theta \cos\varphi\\
          -\sin\varphi
         \end{array}
\right).
 $$
 
 Then:
 $$
 \gamma'(t)= u'X+ u\theta' \frac{\partial X}{\partial \theta} + u\varphi'\frac{\partial X}{\partial \varphi},
 $$
 \begin{align*}
  \gamma''(t)=& \left[u'' - u[\theta']^2 \sin^2\varphi - u[\varphi']^2 \right]X + \left[2u'\varphi' - u[\theta']^2 
\sin\varphi \cos\varphi + u\varphi''\right] \frac{\partial X}{\partial 
\varphi}\\
&+  \left[2u'\theta' + u\theta''+ u\theta'\varphi'\frac{\cos \varphi}{\sin \varphi} + 
u\varphi'\theta'\frac{\cos \varphi}{\sin \varphi} \right]\frac{\partial X}{\partial \theta}
\end{align*}
 and
 $$
 \gamma^\perp = uX - uu'\left[u'X+u\theta'\frac{\partial X}{\partial\theta}+u\varphi'\frac{\partial X}{\partial\varphi}\right].
 $$

In this fashion eq. \eqref{selfshrinker} implies that
\begin{align*}
& u'' - \sin^2 \varphi u [\theta']^2 - u[\varphi']^2 = -u + u[u']^2,\\
& 2u' \theta' + u\theta'' + u \theta' \varphi' \frac{\cos \varphi}{\sin \varphi} + u \varphi' \theta' \frac{\cos \varphi}{\sin \varphi}= u^2 u' \theta',\\
& 2u'\varphi' - u[\theta']^2 \sin \varphi \cos \varphi + u \varphi'' = u^2 u' \varphi'
\end{align*}
and, as we chose a parametrization by arc-length,
$$
[u']^2 + [u\theta']^2\sin^2\varphi + [u\varphi']^2 = 1.
$$
Numerical evaluation of these equations indicate that all self-shrinkers in $\R^3$ lie in planes:
\begin{figure}[h]
\centering
\begin{subfigure}{.5\textwidth}
  \centering
  \includegraphics[width=\linewidth]{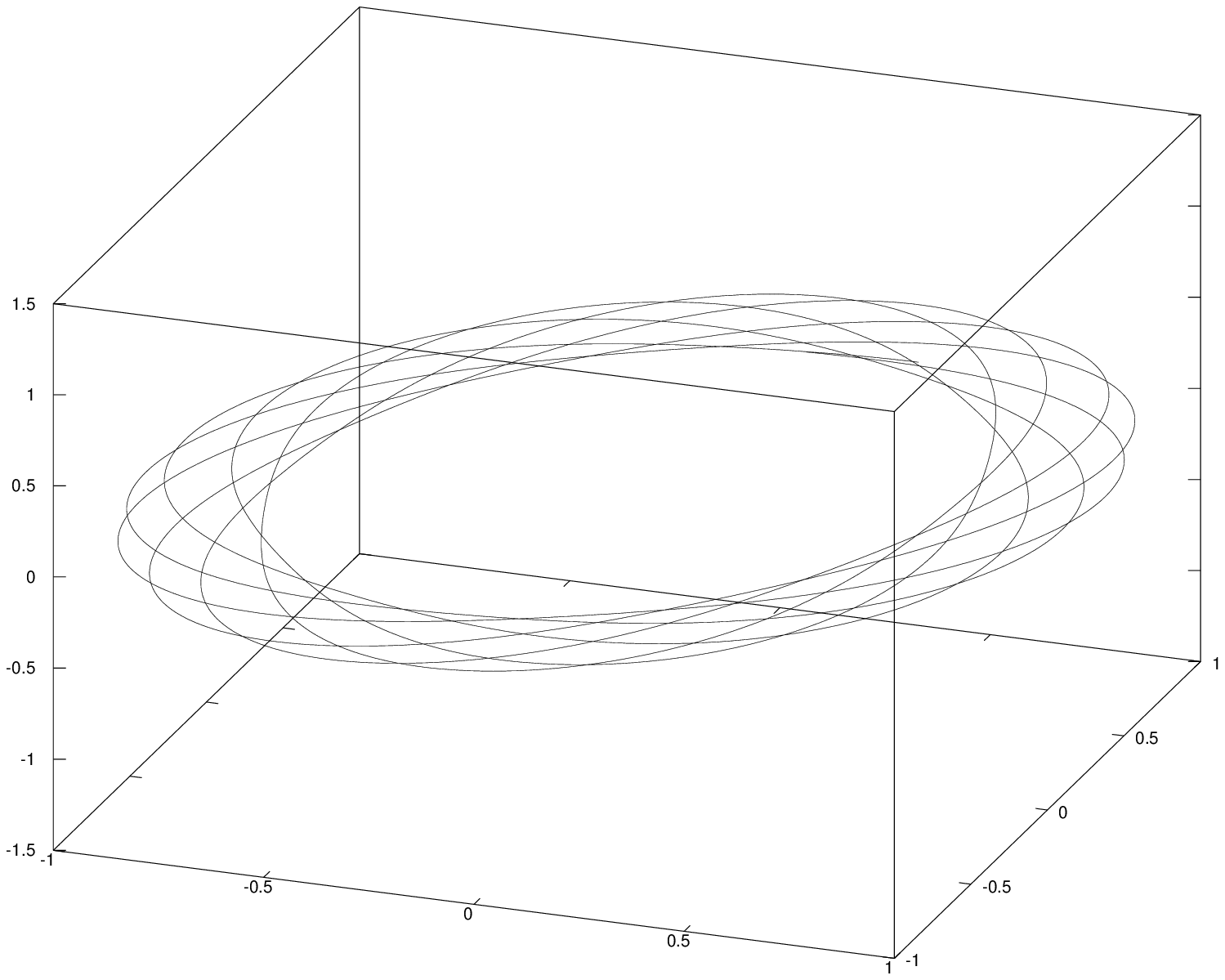}
\end{subfigure}%
\begin{subfigure}{.5\textwidth}
  \centering
  \includegraphics[width=\linewidth]{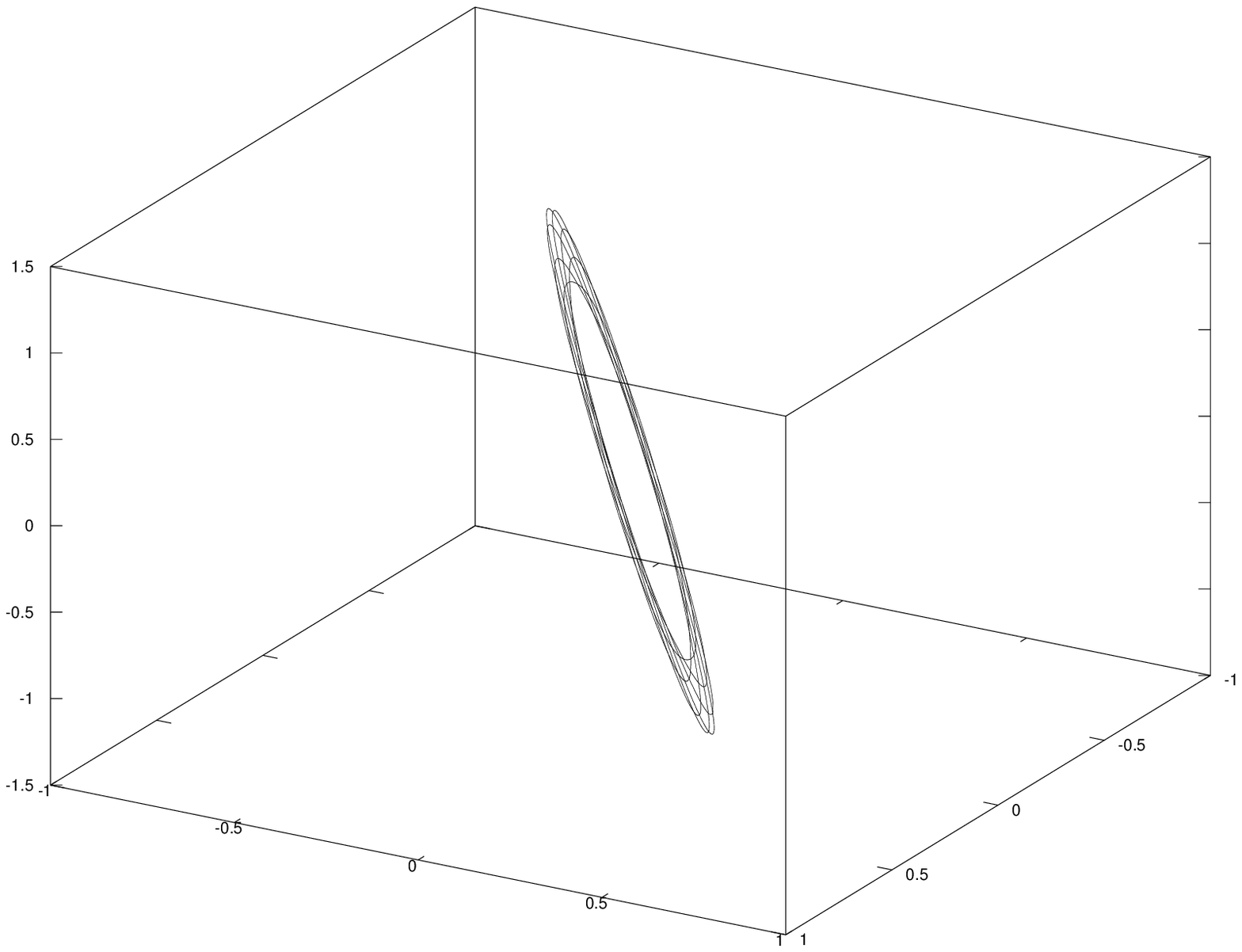}
\end{subfigure}
\caption{Two plots of the same self-shrinker from different angles.}
\end{figure}
\end{section}

\begin{section}{Self-shrinking curves in $\R^n$}

In this section we prove:
\begin{thm}
 Every self-shrinking solution of the curve shortening flow $\gamma: I \to \R^n$ lies in a plane.
\end{thm}
\begin{proof}
 First of all let $\gamma$ be parametrized by arc-length. Then, by eq. \eqref{selfshrinker},
 \begin{align*}
 \gamma'''=&\gamma' - \gamma'+ \langle \gamma, \gamma''  \rangle \gamma' + \langle \gamma, \gamma'  \rangle \gamma''\\
 =& -\langle \gamma, \gamma \rangle \gamma' + \langle \gamma, \gamma' \rangle^2 \gamma' + \langle \gamma, \gamma' 
\rangle \gamma''\\
 =& -\|\gamma''\|^2 \gamma' + \langle \gamma, \gamma' \rangle \gamma''.
 \end{align*}

If $r,s:(a,b) \to \R$ are solutions to
\begin{equation}
\label{oioi} (r\gamma' +s\gamma'')'=0,
\end{equation}
then the vector field $v(t) = r(t)\gamma'(t) +s(t)\gamma''(t)$ over $\gamma(a,b)$ is a constant vector. Note that eq. 
\eqref{oioi} 
implies
$$
r'\gamma' +s'\gamma'' + r\gamma'' +s(-\|\gamma''\|^2 \gamma' + \langle \gamma, \gamma' \rangle \gamma'')=0.
$$
So that, if $\gamma' \neq 0$ and $\gamma'' \neq 0$, $r$ and $s$ satisfy the following O.D.E system:
\begin{equation}
\label{system}\left\{
\begin{matrix}
r'(t) = s(t)(\langle \gamma, \gamma \rangle - \langle \gamma, \gamma' \rangle^2),\\
s'(t) = -s(t)\langle \gamma, \gamma' \rangle -r(t).
\end{matrix}
\right.
\end{equation}

The associated initial value problem has a unique solution for every fixed pair of values for $r(t_0)$ and $s(t_0)$, 
which can be extended for the whole domain of $\gamma$, and any solution to eq. \eqref{system} makes eq. \eqref{oioi} 
hold. Thus $r\gamma'+s\gamma''$ is a constant vector. Further, if the curve defined by $\gamma$ is not a straight line 
or is degenerate to a point, then there is $t_0\in (a,b)$ such that $\gamma'(t_0) \neq 0$ and $\gamma''(t_0) \neq 0$. 
Letting $r(t_0)$ and $s(t_0)$ vary makes 
$v(t_0)=r(t_0)\gamma'(t_0) +s(t_0)\gamma''(t_0)$ equal to any vector in the plane defined by 
$\gamma'(t_0)$, $\gamma''(t_0)$ and the origin.

Furthermore $v(t)=r(t)\gamma'(t) +s(t)\gamma''(t) = r(t_0)\gamma'(t_0) +s(t_0)\gamma''(t_0)=v(t_0)$ for all $t \in 
(a,b)$. Thence the family of $v(t)$ thus obtained spans the same plane for any $t$. There are linearly independent 
vectors in this family, so that $\gamma'(t)$ can be written as a linear combination of two vectors of the like, 
then $\gamma'(t)$ is always on this plane and curve lies in a plane.
\end{proof}
\end{section}

\begin{section}{Self-expanders}
Let $\gamma: I \to \R^2$ be a self-similar expanding solution of the curve shortening flow that is parametrized by 
arc-length. Then
\begin{equation}
\label{selfexpander} \gamma'' = \gamma^\perp = \gamma -\langle \gamma, \gamma' \rangle \gamma'
\end{equation}

Writing $\alpha = \langle \gamma, \gamma\rangle$ we get
\begin{equation}
\label{alpha2}\alpha'' + \frac{(\alpha')^2}{2} -2 \alpha = 2.
\end{equation}

The associated initial value problem admits an unique solution.

\begin{lem}
A solution of eq. \eqref{alpha2} with $0 < \alpha(0)$ and $\alpha'(0)=0$ is strictly positive.
\end{lem}
\begin{proof}
Note that $\alpha$ has a local minimum at 0 and there is no local maximum $t_0$ (or saddle point) with $\alpha(t_0)>0$. 
\end{proof}

If a solution to eq. \eqref{alpha2} is positive, it is possible to define a function $u=\sqrt{\alpha}$ and write the
self-expander in polar coordinates:
\begin{equation}
\label{polarplane2}\gamma (t)= u(t)(\mbox{cos}(\theta(t)), \mbox{sin}(\theta(t))).
\end{equation}

Beyond this $u=\sqrt{\alpha}$ implies
$$
\alpha' = 2uu' \hspace{13mm} \mbox{and} \hspace{13mm} \alpha'' = 2u''u+2(u')^2
$$
so that equation \eqref{alpha2} turns into
\begin{equation}
\label{uuu2} u''u + (u')^2 + [u']^2u^2 - u^2 = 1
\end{equation}

Further equations \eqref{1}, \eqref{2} and \eqref{3} also hold. Therefore equation \eqref{selfexpander} implies that
\begin{align}
\label{A2} u''-u[\theta']^2 = -u[u']^2+u,\\
\label{B2} 2u'\theta'+ u\theta''=-u^2u'\theta'
\end{align}
where $u$ is a known function and $\theta$ is also given by eq. \eqref{C}, because the self-expander is parametrized by 
arc-length.

It is not hard to see that solutions to  equations \eqref{uuu2} and \eqref{C} also satisfy \eqref{A2} and 
\eqref{B2}. Therefore any solutions $u$ and $\theta$ of equations \eqref{uuu2} and \eqref{C} generate self-expanders of 
the curve shortening flow through equation \eqref{polarplane2}. Thus we proved:

\begin{thm}
A curve $\M$ parametrized by $\gamma: I \to \R^2$, $\gamma(t)=\sqrt{\alpha(t)}(\mbox{cos}(\theta(t)), 
\mbox{sin}(\theta(t)))$ is a self-expander of the curve shortening flow if, and only if,
\begin{enumerate}
 \item it is a straight line or
 \item $\alpha(t) > 0$ for all $t \in I$ and 
\begin{align*}
&\alpha'' + \frac{(\alpha')^2}{2} -2 \alpha = 2,\\
&[\theta']^2 = \frac{1-[u']^2}{u^2}.
\end{align*}
\end{enumerate}
\end{thm}

Furthermore, calculations analogous to the previous sections, show that the self-expanders are also 
necessarily planar:

\begin{thm}
 Every self-expanding solution of the curve shortening flow $\gamma: I \to \R^n$ lies in a plane.
\end{thm}
\end{section}

\end{document}